\documentclass{amsart}
\usepackage{amsfonts}
\usepackage{bbm}
\usepackage{mathrsfs}
\usepackage{amsmath}
\usepackage{bigdelim}
\usepackage{multirow}
\usepackage{amssymb}
\usepackage{indentfirst}
\usepackage{CJK}
\usepackage{fancyhdr}
\usepackage{amscd,amssymb,amsmath,graphicx,verbatim,xy}
\usepackage[TS1,OT1,T1]{fontenc}

\newtheorem{theorem}{Theorem}[section]
\newtheorem{lemma}[theorem]{Lemma}

\theoremstyle{definition}

\newtheorem{example}[theorem]{Example}

\theoremstyle{remark}
\newtheorem{remark}[theorem]{Remark}
\numberwithin{equation}{section}

\begin{document}

\title[Continuity of inner-outer factorization and cross sections]{Continuity of inner-outer factorization and cross sections from invariant subspaces to inner functions}

\author{Bingzhe Hou}
\address{Bingzhe Hou, School of Mathematics, Jilin University, 130012, Changchun, P. R. China}
\email{houbz@jlu.edu.cn}

\author{Yue Xin}
\address{Yue Xin, School of Mathematics, Jilin University, 130012, Changchun, P. R. China}
\email{179929393@qq.com}	
\date{}
\thanks{}
\begin{abstract}
Let $H^{\infty}$ be the Banach algebra of bounded analytic functions on the unit open disc $\mathbb{D}$ equipped with the supremum norm. As well known, inner functions play an important role of in the study of bounded analytic functions. In this paper, we are interested in the study of inner functions. Following by the canonical inner-outer factorization decomposition, define $Q_{inn}$ and $Q_{out}$  the maps from $H^{\infty}$ to $\mathfrak{I}$ the set of inner functions and $\mathfrak{F}$ the set of outer functions, respectively. In this paper, we study the $H^{2}$-norm continuity and $H^{\infty}$-norm discontinuity of $Q_{inn}$ and $Q_{out}$ on some subsets of $H^{\infty}$. On the other hand, the Beurling theorem connects invariant subspaces of the multiplication operator $M_z$ and inner functions. We show the nonexistence of continuous cross section from some certain  invariant subspaces to inner functions in the supremum norm. The continuity problem of $Q_{inn}$ and $Q_{out}$ on $\textrm{Hol}(\overline{\mathbb{D}})$, the set of all analytic functions in the closed unit disk, are also considered.
\end{abstract}
\maketitle

\section{Introduction}

Let $\mathbb{D}$ be the unit open disc and $\mathbb{T}$ be the unit circle. Let $H^{\infty}$ be the Banach algebra of bounded analytic functions on $\mathbb{D}$ equipped with the supremum norm ${\Vert}f{\Vert}_{\infty}=\sup\{{\vert}f(z){\vert}; z\in\mathbb{D}\}$. Moreover, denoted by $(H^{\infty})^{-1}$ the set of invertible functions in $H^{\infty}$. Let $L^{\infty}(\mathbb{T})$ (or $L^{\infty}$ in brief) be the collection of all essentially bounded measurable functions on the unit circle $\mathbb{T}$ with regard to the normalized Lebesgue measure of $\mathbb{T}$. $L^{\infty}$ is also a Banach algebra equipped with the the essential supremum norm ${\Vert}g{\Vert}_{L^{\infty}}=\sup_{\xi\in\mathbb{T}}{\rm ess}{\vert} g(z){\vert}$. Then there is an inclusion $i:H^{\infty}\rightarrow L^{\infty}$ defined by
\[
f\rightarrow \widetilde{f}(\mathrm{e}^{\mathbf{i}\theta})=\lim\limits_{r\to 1^{-}} f(r\mathrm{e}^{\mathbf{i}\theta}).
\]
Notice that this inclusion is an isometry, i.e., for any $f\in H^{\infty}$,
\[
\|f\|_{\infty}=\|\widetilde{f}\|_{L^{\infty}}.
\]
Moreover, denoted by $(L^{\infty})^{-1}$ the set of invertible functions in $L^{\infty}$.

A bounded analytic function $u$ on $\mathbb{D}$ is called an inner function if it has unimodular radial limits almost everywhere on $\mathbb{T}$, and denote by $\mathfrak{I}$ the set of all inner functions. As well known, it plays an important role of inner functions in the study of bounded analytic functions, for instance, the Beurling theorem tells us that each invariant subspace $M$ of the classical Hardy space $H^{2}$ under  the multiplication operator $M_z$ there exists an inner function $u$ such that $M=uH^2$.
In addition, a bounded analytic function $F$ on $\mathbb{D}$ is called an outer function if $F$ is a cyclic vector for multiplication operator $M_z$, i.e.,
\[
\bigvee\{z^nF(z); \ n\in \mathbb{N}\}=H^2,
\]
and denote by $\mathfrak{F}$ the set of all outer functions.

For any $f\in H^{\infty}$, there is a canonical inner-outer factorization decomposition $f=uF$, where $F$ is an outer function and $u$ is an inner function which is unique up to a scalar of modulus $1$. Through the present paper, we assume that $u$ is an inner function with $u^{(n_0)}(0)>0$ where $n_0$ is the smallest nonnegative integer such that $u^{(n_0)}(0)$ is non-vanishing, which fixes the choice of the inner function in the inner-outer factorization decomposition and is called a normalized inner function. Furthermore, let $Q_{inn}$ be the mapping
$$
Q_{inn}:H^{\infty}\rightarrow \mathfrak{I}, \ \ \ \ Q_{inn}(f)=u,
$$
and let $Q_{out}$ be the mapping
$$
Q_{out}:H^{\infty}\rightarrow \mathfrak{F}, \ \ \ \ Q_{out}(f)=F.
$$
A natural question is to ask whether the mapping $Q_{inn}$ or $Q_{out}$ is continuous in some certain norm? Unfortunately, V. Kabaila \cite{Kab-1977} showed that neither $Q_{inn}$ nor $Q_{out}$ is continuous in $H^{p}$-norm, $1\le p<\infty$.

In \cite{Doug-1968}, R. Douglas and C. Pearcy made a study on the topology of the invariant subspaces of certain bounded linear operators, where the distance of two invariant subspaces is the norm of the difference of the corresponding orthogonal projections.
If $p$ is a nontrivial projection such that $p(H^2)$ is invariant under $M_z$ multiplication by the coordinate function, then the Beurling theorem gives an inner function
$\varphi$ in $H^{\infty}$ such that $p(H^2)=\varphi H^2$. This $\varphi$ is also unique up to a scalar of modulus $1$, and $p=T_{\varphi}T^*_{\varphi}$, where $T_{\varphi}$ is the Toeplitz operator induced by $\varphi$.
Let $p_t$, $t \in [0, 1]$, be a (operator) norm continuous family of nontrivial projections such that $p_t(H^2)$ is invariant under $M_z$ for each $t$. In this paper, we denote $\varphi_t$ the inner function each $t$ such that $p_t(H^2)=\varphi_t H^2$, and call $\varphi_t$ a cross section of $p_t$ in $\mathfrak{I}$. In addition, we always denote $u_t$ the inner function chosen for each $t$ such that $p_t(H^2)=u_t H^2$ and $u_t^{(n_t)}(0) > 0$, where $n_t$ is the smallest nonnegative integer such that $u_t^{(n_t)}(0)$ is non-vanishing, and call $u_t$ the normalized cross section of $p_t$ in $\mathfrak{I}$.

The components of the set of inner functions has been considered by Herrero in \cite{Her-1974} and \cite{Her-1976}, and by Nestoridis in \cite{Nest79} and \cite{Nest80}.
Let $\mathfrak{I}^{*}$ ($CN^{*}$) be the open set in $H^{\infty}$ of functions of the form $f=uh$, where $u$ is an inner function (Carleson-Newman Blaschke product) and $f\in (H^{\infty})^{-1}$. Notice that a function $f$ belongs to $\mathfrak{I}^{*}$ if and only if  $h\in H^{\infty}$ and $\inf\limits_{\xi\in\mathbb{T}}|\widetilde{f}(\xi)|>0$, that is $f\in H^{\infty}$ and $\widetilde{f}\in(L^{\infty})^{-1}$. A result of Laroco \cite{Lar-1991} asserts that the set $\mathfrak{I}^{*}$ is dense in $H^{\infty}$. A. Nicolau and D. Su\'{a}rez \cite{NS-2010} studied the connected components of $\mathfrak{I}^{*}$ and $CN^{*}$.

In this paper, we will consider the subsets of $\mathfrak{I}^{*}$ with the same multiplicity of the zero point at $0$.
For any $f\in H^{\infty}$, denote by $\textrm{Mul}_0(f)$ be the multiplicity of the zero point of $f$ at $0$, more precisely,
\[
\textrm{Mul}_0(f)=\inf\{n; f^{(n)}(0)\neq 0\}.
\]
Furthermore, we denote
\[
\mathfrak{I}^{*}_n=\{f\in\mathfrak{I}^{*}; \textrm{Mul}_0(f)=n\}, \ \  \ \text{for any} \ n=0,1,2,\ldots.
\]
Notice that
\[
\mathfrak{I}^{*}_n=\{z^nf; f\in\mathfrak{I}^{*}_0\}.
\]

In the present paper, we study the $H^{2}$-norm continuity of $Q_{inn}$ and $Q_{out}$ on $\mathfrak{I}^{*}_n$ and $\mathfrak{I}^{*}$ in section $2$. In section $3$, we study the the $H^{\infty}$-norm discontinuity of $Q_{inn}$ and $Q_{out}$ on $\mathfrak{I}^{*}_n$, and show the nonexistence of continuous cross section from invariant subspaces to inner functions under essential supremum norm. In the final section $4$, we also consider the continuity problem of $Q_{inn}$ and $Q_{out}$ on $\textrm{Hol}(\overline{\mathbb{D}})$, the set of all analytic functions in the closed unit disk.

\section{Continuity of inner-outer factorization}

Following from the proof of Theorem $7.1$ in \cite{Garn-1981}, one can see that if $p\mathbf{1}\neq 0$, where $\mathbf{1}$ is the constant function $1$ in $H^{2}$, then $u=\frac{p\mathbf{1}}{\|p\mathbf{1}\|}$ is the normalized inner function of an orthogonal projection $p$ on an invariant subspace of $M_z$. Consequently, $u$ is $H^{2}$-norm continuous with respect to $p$ for $p\mathbf{1}(0)\neq 0$.

\begin{lemma} [\cite{Garn-1981}, p.~80]\label{pG}
Suppose that $p_t\mathbf{1}(0)\neq 0$ for all $t$. The normalized cross section $u_t$ of $p_t$ in $\mathfrak{I}$ is continuous in $H^{2}$-norm.
\end{lemma}

Let $\mathfrak{M}$ and $\mathfrak{N}$ be two nontrivial subspaces in a Hilbert space $\mathcal{H}$, $p_\mathfrak{M}$ and $p_\mathfrak{N}$ be the orthogonal projections on $\mathfrak{M}$ and $\mathfrak{N}$, respectively.
The gap (aperture) between subspaces $\mathfrak{M}$ and $\mathfrak{N}$ defined as, e.g., \cite{Doug-1968} and \cite{Kato-1995},
$$
{\rm gap}(\mathfrak{M} ,\mathfrak{N})=\|p_\mathfrak{M}-p_\mathfrak{N}\|=\max\{\|p_\mathfrak{M}p_{\mathfrak{N}^{\bot}}\|, \|p_\mathfrak{N}p_{\mathfrak{M}^{\bot}}\|\}
$$
is used to measure the distance between subspaces.

The maximal angle $\theta_{\rm max}(\mathfrak{M} ,\mathfrak{N})$ between $\mathfrak{M}$ and $\mathfrak{N}$ was introduced in \cite{KKM-1948} and is defined as the angle in
$[0, \pi/2]$ given by
$$
\sin \theta_{\rm max}(\mathfrak{M} ,\mathfrak{N})= \sup\limits_{x\in \mathfrak{M}, \|x\|=1}{\rm dist}(x, \mathfrak{N})=\sup\limits_{x\in \mathfrak{M}, \|x\|=1}\sqrt{1-\|p_\mathfrak{N}(x)\|^2}.
$$

\begin{lemma} [\cite{BS-2010} or \cite{KJA-2010}]\label{top}
	$$
	\|p_\mathfrak{M}-p_\mathfrak{N}\|=\max\{\sin \theta_{\rm max}(\mathfrak{M} ,\mathfrak{N}), \sin \theta_{\rm max}(\mathfrak{N} ,\mathfrak{M})\}.
	$$
\end{lemma}

\begin{theorem}\label{cont}
For any nonnegative integer $n$, the maps $Q_{inn}:(\mathfrak{I}^{*}_n, \|\cdot\|_{\infty})\rightarrow (\mathfrak{I}, \|\cdot\|_{2})$ and $Q_{out}:(\mathfrak{I}^{*}_n, \|\cdot\|_{\infty})\rightarrow (\mathfrak{F}, \|\cdot\|_{2})$ are continuous in the $H^{2}$-norm.
\end{theorem}

\begin{proof}
Notice that for any nonnegative integer $n$,
\[
\mathfrak{I}^{*}_n=\{z^nf; f\in\mathfrak{I}^{*}_0\}.
\]
It suffices to consider the case of $n=0$.

Given any element $f$ in $\mathfrak{I}^{*}_0$. There is a canonical decomposition $f=uF$, where $u=Q_{inn}(f)$ is a normalized inner function and $F=Q_{out}(f)$ is an outer function. Since $F$ is an invertible element in $H^{\infty}$, there exist two positive constants $c_{1}$ and $C_{1}$, such that
\begin{equation*}
0<c_{1}\le \|F\|_{\infty}\le C_{1}<+\infty.
\end{equation*}
Similarly, for an element $g$ in $\mathfrak{I}^{*}_0$, we may write $g=vG$,  where $v=Q_{inn}(g)$ and $G=Q_{out}(g)$. Then, there are two positive constants $c_{2}$ and $C_{2}$, such that
\begin{equation*}
0<c_{2}\le \|G\| _{\infty}\le C_{2}<+\infty.
\end{equation*}
Now denote
\[
c=\min\{c_1,c_2\} \ \ \ \ \text{and} \ \ \ \ C=\max\{C_1,C_2\}.
\]
We will prove the continuity of $Q_{inn}:(\mathfrak{I}^{*}_0, \|\cdot\|_{\infty})\rightarrow (\mathfrak{I}, \|\cdot\|_{2})$ and $Q_{out}:(\mathfrak{I}^{*}_0, \|\cdot\|_{\infty})\rightarrow (\mathfrak{F}, \|\cdot\|_{2})$, respectively.

{\bf Part (1) The $H^{2}$-norm continuity of $Q_{inn}$ on $\mathfrak{I}^{*}_0$.}

Denote $\mathfrak{M}$ the nontrivial invariant subspace $u(H^{2})$ of multiplication operator $M_z$  and denote $v(H^{2})$ by $\mathfrak{N}$. Let $p_\mathfrak{M}$ and $p_\mathfrak{N}$ be the corresponding nontrivial orthogonal projections from $H^{2}$ onto $\mathfrak{M}$ and $\mathfrak{N}$, respectively. Since $F$ and $G$ are invertible in $H^{\infty}$, we have
\[
p_\mathfrak{M} (H^2)=uH^2=fH^2 \ \ \ \ \text{and} \ \ \ \ p_\mathfrak{N} (H^2)=vH^2=gH^2.
\]

For any element $\varphi\in \mathfrak{M}$  with $\|\varphi\|_{2}=1$, there is an element $h\in H^2$ such that $\varphi=uh$. Moreover, $\|h\|_{2}=1$ and $\varphi=uF\cdot\frac{h}{F}=f\cdot\frac{h}{F}$. Since $g\cdot\frac{h}{F}$ is an element in the subspace $\mathfrak{N}$, we have
\begin{align*}
\textrm{dist}(\varphi,\mathfrak{N})\triangleq& \inf\{\|\varphi-w\|_2;~w\in \mathfrak{N}\} \\
\le& \|\varphi-g\cdot\frac{h}{F}\|_2 \\
=&\|f\cdot\frac{h}{F}-g\cdot\frac{h}{F}\|_{2}\\
\le&\|f-g\|_{\infty}\|\frac{1}{F}\|_{\infty}\|h\|_{2}\\
\le& \frac{1}{c}\cdot\|f-g\|_{\infty}.
\end{align*}
By the arbitrariness of $\varphi$, we have
\begin{align*}
	\sup_{\varphi\in \mathfrak{M},\ \|\varphi\|_2=1} \textrm{dist}(\varphi,\mathfrak{N})\le \frac{1}{c}\cdot\|f-g\|_{\infty}.
\end{align*}
Similarly, we have
 \begin{align*}
 		\sup_{\psi\in \mathfrak{N},\ \|\psi\|_2=1} \textrm{dist}(\psi,\mathfrak{M})\le \frac{1}{c}\cdot\|f-g\|_{\infty}.
 \end{align*}

By Lemma \ref{top}, one can see that
\begin{equation}\label{eqpmpn}
		\|p_\mathfrak{M}-p_\mathfrak{N}\|=\max\{\sin \theta_{\rm max}(\mathfrak{M}, \mathfrak{N}), \sin \theta_{\rm max}(\mathfrak{N},\mathfrak{M})\}\le \frac{1}{c}\cdot\|f-g\|_{\infty}.
\end{equation}

Given any $\epsilon>0$. It follows from Lemma \ref{pG} that there is a positive number $\epsilon'$ such that  when $\|p_\mathfrak{M}-p_\mathfrak{N}\|\le \epsilon'$, we have $\|u-v\|_{2}\le\epsilon$.

Let $\delta=c\epsilon'$. Then, when $\|f-g\|_{\infty}\le\delta$, we have
\begin{align*}
	\|p_\mathfrak{M}-p_\mathfrak{N}\|\le \frac{\delta}{c}= \frac{c\epsilon'}{c}=\epsilon'
\end{align*}
and consequently
\begin{align*}
	\|u-v\|_{2}\le\epsilon.
\end{align*}
Therefore, the map $Q_{inn}:(\mathfrak{I}^{*}_0, \|\cdot\|_{\infty})\rightarrow (\mathfrak{I}, \|\cdot\|_{2})$ is continuous in $H^{2}$-norm.

{\bf Part (2) The $H^{2}$-norm continuity of $Q_{out}$ on $\mathfrak{I}^{*}_0$.}

Since
\begin{align*}
f-g=&uF-vG\\
=&uF-vF+vF-vG\\
=&F(u-v)+v(F-G),
\end{align*}
we have
	\begin{align*}
\|F-G\|_{2}&=\|v(F-G)\|_{2} \\
&\le\|f-g\|_{2} +\|F(u-v)\|_{2} \\
&\le\|f-g\|_{\infty} +\|F\|_{\infty} \|u-v\|_{2} \\
&\le\|f-g\|_{\infty} +C\|u-v\|_{2}.
    \end{align*}
Given any $\eta>0$. It follows from Part (1) that there is a positive number $\delta<\frac{\eta}{2}$ such that  when $\|f-g\|_{\infty}\le \delta$, we have $\|u-v\|_2\le\frac{\eta}{2C}$. Then, we have
\[
	\|F-G\|_{2}\le \delta+C\cdot\frac{\eta}{2C}\le\frac{\eta}{2}+\frac{\eta}{2}=\eta.
\]
Therefore, the map $Q_{out}:(\mathfrak{I}^{*}_0, \|\cdot\|_{\infty})\rightarrow (\mathfrak{F}, \|\cdot\|_{2})$ is continuous in $H^{2}$-norm.
\end{proof}

In the above theorem, the setting of $\mathfrak{I}^{*}_n$ is necessary. Neither $Q_{inn}:(H^{p}, \|\cdot\|_{p})\rightarrow (\mathfrak{I}, \|\cdot\|_{p})$ nor $Q_{out}:(H^{p}, \|\cdot\|_{p})\rightarrow (\mathfrak{F}, \|\cdot\|_{p})$ is continuous, $1\le p\le\infty$. Kabaila \cite{Kab-1977} proved the discontinuity for the case of $1\le p<\infty$. The discontinuity for the case of $p=\infty$ could be found in Corollary $3$ in \cite{Nak-1996} by Nakazi. Now, we will give a simple example to show that neither $Q_{inn}:(\mathfrak{I}^{*}, \|\cdot\|_{\infty})\rightarrow (\mathfrak{I}, \|\cdot\|_{p})$ nor $Q_{out}:(\mathfrak{I}^{*}, \|\cdot\|_{\infty})\rightarrow (\mathfrak{F}, \|\cdot\|_{p})$ is continuous for any $1\le p\le\infty$, which implies that $\mathfrak{I}^{*}_n$ could not be extended to $\mathfrak{I}^{*}$ in the Theorem \ref{cont}.

\begin{example}\label{I*n}
Let $f_t(z)=\frac{t-z}{1-tz}$ for $t\in[-1,1]$. Obviously, $f_t(z)$ is a continuous path in $\mathfrak{I}^{*}$. However,
\[
Q_{inn}(f_t)=\left\{\begin{array}{cc}
-\frac{t-z}{1-tz} \ \ \ &\text{for} \ t\in[-1,0], \\
\frac{t-z}{1-tz} \ \ \ &\text{for} \ t\in(0,1].
\end{array}\right.
\]
and
\[
Q_{out}(f_t)=\left\{\begin{array}{cc}
-1 \ \ \ &\text{for} \ t\in[-1,0], \\
1 \ \ \ &\text{for} \ t\in(0,1].
\end{array}\right.
\]
Then, neither $Q_{inn}(f_t)$ nor $Q_{out}(f_t)$ is continuous at $t=0$ in $H^{p}$-norm, for any $1\le p\le\infty$.
\end{example}

\section{Nonexistence of continuous cross sections from invariant subspaces to inner functions in the supremum norm}

In this section, we will show the discontinuity of $Q_{inn}:(\mathfrak{I}^{*}_n, \|\cdot\|_{\infty})\rightarrow (\mathfrak{I}, \|\cdot\|_{\infty})$ and $Q_{out}:(\mathfrak{I}^{*}_n, \|\cdot\|_{\infty})\rightarrow (\mathfrak{F}, \|\cdot\|_{\infty})$. Moreover, we will illustrate the nonexistence of continuous cross section from invariant subspaces to inner functions in the supremum norm. To prove those, we need some lemmas as preliminaries.

Firstly, Nicolau and Su\'{a}rez have obtained a result with regard to path in $\mathfrak{I}^{*}$ in the supremum norm.

\begin{lemma}[Proposition 4.5 in \cite{NS-2010}]\label{NS}
Let $f$, $g\in\mathfrak{I}^{*}$. Then there is a normalized inner function $b$ (in fact, it is a CNBP) such that $bf$ and $bg$ can be joined by a polygonal path contained in $\mathfrak{I}^{*}$ in the supremum norm. Moreover, if $f$, $g\in CN^{*}$, then $b$ can be chosen such that $bf$ and $bg$ can be joined by a polygonal path contained in $CN^{*}$.
\end{lemma}

For a polygonal path, we have an intuitive observation as follows.
\begin{lemma}\label{polypath}
Let $\gamma(t):[0,1]\rightarrow \mathbb{C}$ be a polygonal path. Then for any $\mathrm{e}^{\mathbf{i}\eta}$ which is not parallel to any segment of the polygonal path $\gamma$, there exists $\epsilon_0>0$ such that for any $0<\epsilon\le\epsilon_0$, the path $\gamma(t)+\epsilon\mathrm{e}^{\mathbf{i}\eta}$ does not pass through $0$.
\end{lemma}
\begin{proof}
The proof is simple. Since a polygonal path is composed of finite segments, one can see the conclusion from the following figure.
\begin{figure}[htbp]
\centering
\includegraphics[width=0.7\textwidth]{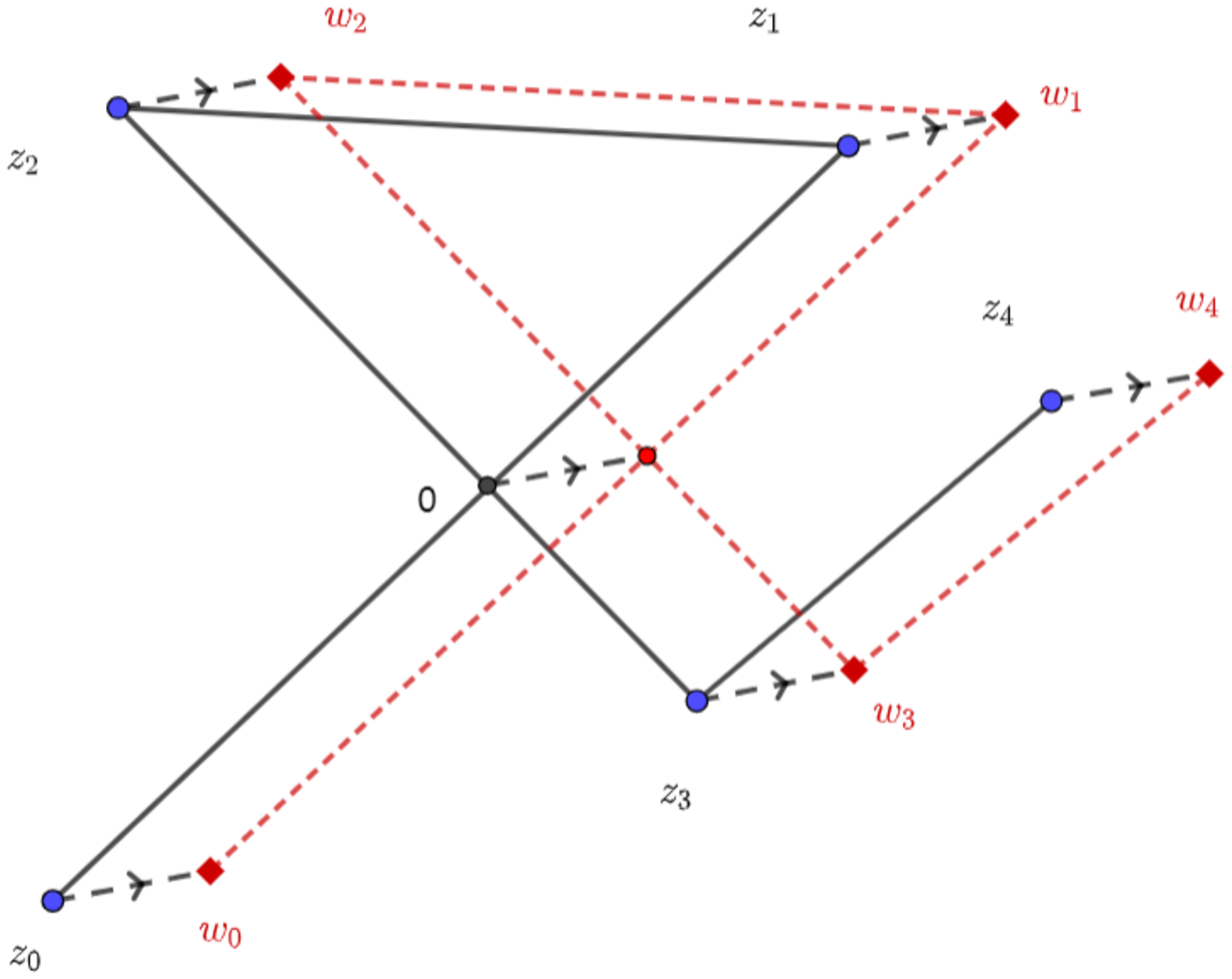}
\caption{}
\label{figure1}
\end{figure}
\end{proof}

Following from the above simple lemma, we could strengthen the conclusion of Nicolau and Su\'{a}rez (Lemma \ref{NS}) from $\mathfrak{I}^{*}$ to $\mathfrak{I}^{*}_0$.

\begin{lemma}\label{trans}
Let $h_0, h_1\in (H^{\infty})^{-1}$. Then there exists a normalized inner function $\varphi$ such that $\varphi h_0$ and $\varphi h_1$ can be joined by a path contained in $\mathfrak{I}^{*}_0$ in the supremum norm.
\end{lemma}
\begin{proof}
By Lemma \ref{NS}, there is a normalized inner function $b$ such that $bh_0$ and $bh_1$ can be joined by a polygonal path $f_t(z)$ in $\mathfrak{I}^{*}$, $t\in[0,1]$. Write $f_t=u_th_t$, where $u_t$ is the normalized inner function part of $f_t$ and $h_t$ is the outer function part of $f_t$.

Whenever $b(0)=0$ or $b(0)>0$, for a fixed number $r\in(0,1)$, we have $b(0)+r>0$. Let $\varphi(z)=\frac{b(z)+r}{1+rb(z)}$. Then, $\varphi(0)>0$ and
\[
g_t(z)=\frac{b(z)+(1-t)r}{1+(1-t)rb(z)}, \ \ \ \ t\in [0,1],
\]
is a continuous path from $\varphi$ to $b$ in $\mathfrak{I}$ in the supremum norm, and
\[
\widehat{g}_t(z)=\frac{b(z)+tr}{1+trb(z)}, \ \ \ \ t\in [0,1],
\]
is the inverse path of $g_t$. Furthermore, let
\[
\widetilde{f}_t(z)=\left\{\begin{array}{cc}
g_{4t}(z)h_0(z) \ \ \ &\text{for} \ t\in[0,\frac{1}{4}], \\
f_{4t-1}(z) \ \ \ &\text{for} \ t\in[\frac{1}{4},\frac{1}{2}], \\
\widehat{g}_{2t-1}(z)h_1(z) \ \ \ &\text{for} \ t\in[\frac{1}{2},1].
\end{array}\right.
\]
Then, $\widetilde{f}_t(z)$ is a path from $\varphi h_0$ to $\varphi h_1$ in $\mathfrak{I}^{*}$. Since the unit closed interval is compact, there is a positive number $\epsilon_0$ such that
\[
0<\epsilon_0<\inf\limits_{\xi\in\mathbb{T}}\textrm{ess}\widetilde{f}_t(\xi) \ \ \  \ \text{for all} \ t\in[0,1].
\]
Moreover, let $\gamma(t)=\widetilde{f}_t(0)$, $t\in[0,1]$. By Lemma \ref{NS} and the construction of $\widetilde{f}_t(z)$, one can see that $\gamma(t)$ is a polygonal path in $\mathbb{C}$. By Lemma \ref{polypath}, there exist a number $\eta$ and a positive number $\epsilon$ with $0<\epsilon<\min\{\epsilon_0, |\varphi(0)h_0(0)|, |\varphi(0)h_1(0)|\}$ such that the path $\gamma(t)+\epsilon\mathrm{e}^{\mathbf{i}\eta}$ does not pass through $0$. Consequently,
\[
\widetilde{f}_t(z)+\epsilon\mathrm{e}^{\mathbf{i}\eta}
\]
is a path from $\varphi h_0+\epsilon\mathrm{e}^{\mathbf{i}\eta}$ to $\varphi h_1+\epsilon\mathrm{e}^{\mathbf{i}\eta}$ in $\mathfrak{I}^{*}_0$. Notice that
$\varphi h_0+t\epsilon\mathrm{e}^{\mathbf{i}\eta}$, $t\in[0,1]$, is a path from $\varphi h_0$ to $\varphi h_0+\epsilon\mathrm{e}^{\mathbf{i}\eta}$ in $\mathfrak{I}^{*}_0$, and $\varphi h_1+(1-t)\epsilon\mathrm{e}^{\mathbf{i}\eta}$, $t\in[0,1]$, is a path from $\varphi h_1+\epsilon\mathrm{e}^{\mathbf{i}\eta}$ to $\varphi h_1$ in $\mathfrak{I}^{*}_0$.
Then, we could obtain a path from $\varphi h_0$ to $\varphi h_1$ in $\mathfrak{I}^{*}_0$.
\end{proof}

\begin{theorem}\label{nocont}
Neither $Q_{inn}:(\mathfrak{I}^{*}_n, \|\cdot\|_{\infty})\rightarrow (\mathfrak{I}, \|\cdot\|_{\infty})$ nor $Q_{out}:(\mathfrak{I}^{*}_n, \|\cdot\|_{\infty})\rightarrow (\mathfrak{F}, \|\cdot\|_{\infty})$ are continuous in the supremum norm.
\end{theorem}

\begin{proof}
It suffices to consider the case of $n=0$.

As well known, there exist $h_0, h_1\in (H^{\infty})^{-1}$ such that they can not be joined by a path in $(H^{\infty})^{-1}$ (see \cite{Gam-1984} or \cite{Dav-1988} for example). More precisely, one can choose
\[
h_0=\mathbf{1} \ \ \ \  \text{and} \ \ \ \ h_1(z)=\mathrm{e}^{\frac{2\mathbf{i}}{\pi}\log\frac{1+z}{1-z}}
\]
as required. By Lemma \ref{trans}, there exists a normalized inner function $\varphi$ such that $\varphi h_0$ and $\varphi h_1$ can be joined by a path $f_t(z)$, $t\in[0,1]$, contained in $\mathfrak{I}^{*}_0$ in the supremum norm.

Notice that $Q_{out}(f_t)$ belongs to $(H^{\infty})^{-1}$ for each $t\in[0,1]$. If $Q_{out}:(\mathfrak{I}^{*}_n, \|\cdot\|_{\infty})\rightarrow (\mathfrak{F}, \|\cdot\|_{\infty})$ is continuous, $Q_{out}(f_t)$ is a path from $h_0$ to $h_1$ in $(H^{\infty})^{-1}$. That is a contradiction.

Now suppose that $Q_{inn}:(\mathfrak{I}^{*}_n, \|\cdot\|_{\infty})\rightarrow (\mathfrak{I}, \|\cdot\|_{\infty})$ is continuous. Then, $Q_{inn}(f_t)$ is a path in $\mathfrak{I}$. Let
\[
C\triangleq\sup\limits_{t\in[0,1]}\|f_t\|_{\infty}=\sup\limits_{t\in[0,1]}\|Q_{out}(f_t)\|_{\infty}.
\]
For any $s,t\in[0,1]$, since
\begin{align*}
f_s-f_t=&u_sF_s-u_tF_t\\
=&u_sF_s-u_tF_s+u_tF_s-u_tF_t\\
=&F_s(u_s-u_t)+u_t(F_s-F_t),
\end{align*}
we have
	\begin{align}\label{innout1}
\|F_s-F_t\|_{\infty}&=\|u_t(F_s-F_t)\|_{\infty} \\ \label{innout2}
&\le\|f_s-f_t\|_{\infty} +\|F_s\|_{\infty} \|u_s-u_t\|_{\infty} \\ \label{innout3}
&\le\|f_s-f_t\|_{\infty} +C\|u_s-u_t\|_{\infty}.
    \end{align}
Then, the continuity of $f_t$ and $u_t=Q_{inn}(f_t)$ implies the continuity of $F_t=Q_{out}(f_t)$. That means $F_t=Q_{out}(f_t)$ is a path from $h_0$ to $h_1$ in $(H^{\infty})^{-1}$, which is a contradiction.
\end{proof}

\begin{theorem}\label{nocrosssection}
There exists a norm continuous family of nontrivial projections $p_t$ such that $p_t(H^2)$ is invariant for each $t$, $t \in [0, 1]$, such that there is no continuous cross section of $p_t$ in $\mathfrak{I}$ in the supremum norm.
\end{theorem}

\begin{proof}
Similar to the proof of the previous theorem, let
\[
h_0=\mathbf{1} \ \ \ \  \text{and} \ \ \ \ h_1(z)=\mathrm{e}^{\frac{2\mathbf{i}}{\pi}\log\frac{1+z}{1-z}},
\]
and let $f_t(z)$, $t\in[0,1]$, be a path from $\varphi h_0$ to $\varphi h_1$ in $\mathfrak{I}^{*}_0$ for some certain normalized inner function $\varphi$. Furthermore, let $\mathfrak{M}_t$ be the subspace $f_t\mathbb{H}^2$, and let $p_t$ be the orthogonal projection onto $\mathfrak{M}_t$. Following from the inequality (\ref{eqpmpn}), the continuity of $f_t$ in the supremum norm implies the continuity of $p_t$ in the operator norm.

Suppose that $\varphi_t(z)$ is a supremum norm continuous cross section of $p_t$ in $\mathfrak{I}$, which is not necessary to be normalized. Write $f_t(z)=\varphi_t(z)G_t(z)$ for $t\in[0,1]$. As the same as inequalities (\ref{innout1})-(\ref{innout3}), one can see for any $s,t\in[0,1]$,
\[
\|G_s-G_t\|_{\infty}\le\|f_s-f_t\|_{\infty} +C\|\varphi_s-\varphi_t\|_{\infty},
\]
where $C=\sup\limits_{t\in[0,1]}\|f_t\|_{\infty}$ is a positive constant. Then, $G_t$ is a path from $h_0$ to $h_1$ in $(H^{\infty})^{-1}$, which is a contradiction. Therefore, there is no continuous cross section of such $p_t$ in $\mathfrak{I}$ in the supremum norm.
\end{proof}

\begin{remark}
The above Theorem \ref{nocrosssection} tells us that the supremum norm is quite different from $\mathbb{H}^2$-norm. Replace $\mathbb{H}^2$-norm by the supremum norm in Lemma \ref{pG}, the conclusion will no longer hold true, even if the normalized restriction is removed.
\end{remark}

\section{On analytic functions in the closed disk}

In this section, we will consider $\textrm{Hol}(\overline{\mathbb{D}})$ instead of $\mathfrak{I}^{*}$. For any nonnegative integer $n$, denote
\[
\textrm{Hol}_n(\overline{\mathbb{D}})=\{f\in\textrm{Hol}(\overline{\mathbb{D}}); \textrm{Mul}_0(f)=n\}.
\]

\begin{theorem}
	Given any nonnegative integer $n$. The map $Q_{out}:(\textrm{Hol}_n(\overline{\mathbb{D}}), \ \|\cdot\|_{\infty})\rightarrow (\mathfrak{F}, \|\cdot\|_{\infty})$ is continuous in the supremum norm, but not the map $Q_{inn}:(\textrm{Hol}_n(\overline{\mathbb{D}}), \|\cdot\|_{\infty})\rightarrow (\mathfrak{I}, \|\cdot\|_{\infty})$. Both $Q_{inn}:(\textrm{Hol}_n(\overline{\mathbb{D}}), \|\cdot\|_{\infty})\rightarrow (\mathfrak{I}, \|\cdot\|_{2})$ and $Q_{out}:(\textrm{Hol}_n(\overline{\mathbb{D}}), \|\cdot\|_{\infty})\rightarrow (\mathfrak{F}, \|\cdot\|_{2})$ are continuous in the $H^{2}$-norm.
\end{theorem}

\begin{proof}	
Notice that for any nonnegative integer $n$,
\[
\textrm{Hol}_n(\overline{\mathbb{D}})=\{z^nf; f\in\textrm{Hol}_0(\overline{\mathbb{D}})\}.
\]
It suffices to consider the case of $n=0$.

{\bf (1) The continuity of $Q_{out}:(\textrm{Hol}_0(\overline{\mathbb{D}}), \ \|\cdot\|_{\infty})\rightarrow (\mathfrak{F}, \|\cdot\|_{\infty})$}

Given any element $f$ in $\textrm{Hol}_0(\overline{\mathbb{D}})$. Then $f$ has finite zero points in $\overline{\mathbb{D}}\setminus \{0\}$. More precisely, $f$ has zero points $\{\alpha_{k}\}_{k=1}^{n}$ in $\mathbb{D}$ and $\{z_{s}\}_{s=1}^{m}$ in $\partial \mathbb{D}$. Consequently, $f$ has a canonical decomposition,
	\begin{equation*}
		f=B_f F_f=\prod\limits_{k=1}^{n}\frac{|\alpha_{k}|}{\alpha_{k}}\frac{\alpha_{k}-z}{1-\overline{\alpha_{k}}z}\cdot F_0 \cdot\prod\limits_{s=1}^{m}(z_{s}-z),
	\end{equation*}
	in which $Q_{inn}(f)=B_f=\prod\limits_{k=1}^{n}\frac{|\alpha_{k}|}{\alpha_{k}}\frac{\alpha_{k}-z}{1-\overline{\alpha_{k}}z}$ is the normalized inner function part of $f$ and $Q_{out}(f)=F_f=F_0 \prod\limits_{s=1}^{m}(z_{s}-z)$ is the outer function part of $f$. Denote $F_1=\prod\limits_{s=1}^{m}(z_{s}-z)$ and then $F_f=F_0 F_1$. Since $F_{0}$ is an invertible element in $H^{\infty}$, there exist two positive numbers $C_1$ and $C_2>1$ such that
	\begin{equation*}
		0<C_{1}\le\|F_0\|_{\infty}\le C_{2}<+\infty.
	\end{equation*}

	By Rouch\'{e}'s theorem, if $g\in \textrm{Hol}_{0}(\overline{\mathbb{D}})$ is a small perturbation of $f$, then $g$ and $f$ have the same number of zeros in a small neighborhood of $\overline{\mathbb{D}}$. In fact, each zero point of $g$ is close to a zero point of $f$.
	Furthermore, we may assume that $g$ has zero points $\{\widetilde{\alpha_{k}}\}_{k=1}^{n}$ in $\mathbb{D}$,  $\{\widetilde{z_{s}}\}_{s=1}^{m_1}$ in $\mathbb{D}$ and $\{\widetilde{z_{s}}\}_{s=m_1 +1}^{m}$ out of $\mathbb{D}$, which are close to $\{\alpha_{k}\}_{k=1}^{n}$, $\{z_{s}\}_{s=1}^{m_1}$ and $\{z_{s}\}_{s=m_1+1}^{m}$, respectively. Then, denote
	\begin{equation*}
		\widetilde{B}_{f}=\prod_{k=1}^{n}\frac{|\widetilde{\alpha_{k}}|}{\widetilde{\alpha_{k}}}\frac{\widetilde{\alpha_{k}}-z}{1-\overline{\widetilde{\alpha_{k}}}z},
	\end{equation*}
	\begin{equation*}
		G^{inn}_{F_1}=\prod_{s=1}^{m_1}(\widetilde{z_{s}}-z),
	\end{equation*}
	and
	\begin{equation*}
		G^{out}_{F_1}=\prod_{s=m_1 +1}^{m}(\widetilde{z_{s}}-z).
	\end{equation*}
	Moreover, denote
	\begin{equation*}
		\widehat{G^{inn}_{F_1}}=\prod_{s=1}^{m_1}\frac{|\widetilde{z_{s}}|}{\overline{\widetilde{z_{s}}}}(1-\overline{\widetilde{z_{s}}}z)
	\end{equation*}
	and
	\begin{equation*}	\widetilde{B}_{F}=\frac{G^{inn}_{F_1}}{\widehat{G^{inn}_{F_1}}}=\prod_{s=1}^{m_1}\frac{\overline{\widetilde{z_{s}}}}{|\widetilde{z_{s}}|}\frac{(\widetilde{z_{s}}-z)}{(1-\overline{\widetilde{z_{s}}}z)}.
	\end{equation*}
	Then, $Q_{inn}(g)=B_g=\widetilde{B}_{f}\widetilde{B}_{F}$ is the normalized inner function part of $g$ and $Q_{out}(g)=F_g=\widehat{G^{inn}_{F_1}}G^{out}_{F_1}G_0$ is the outer function part of $g$, in which $G_0$ is an invertible element in $H^{\infty}$. So one could write $g$ as follows
	\begin{align*}
		g&=B_gF_g\\
		&=\widetilde{B}_{f}\widetilde{B}_{F}\widehat{G^{inn}_{F_1}}G^{out}_{F_1}G_0\\
		&=\prod_{k=1}^{n}\frac{|\widetilde{\alpha_{k}}|}{\widetilde{\alpha_{k}}}\frac{\widetilde{\alpha_{k}}-z}{1-\overline{\widetilde{\alpha_{k}}}z}\prod_{s=1}^{m_1}\frac{\overline{\widetilde{z_{s}}}}{|\widetilde{z_{s}}|}\frac{(\widetilde{z_{s}}-z)}{(1-\overline{\widetilde{z_{s}}}z)}\prod_{s=1}^{m_1}\frac{|\widetilde{z_{s}}|}{\overline{\widetilde{z_{s}}}}(1-\overline{\widetilde{z_{s}}}z)\prod_{s=m_1 +1}^{m}(\widetilde{z_{s}}-z)G_0.
	\end{align*}

	%2.Ô¼¶¨$\epsilon$
	
	Given any $\epsilon>0$. Let
	\begin{equation*}
		M=\epsilon+\max\{\|\prod\limits_{s=1}^{m_1}(z_{s}-z)\|_{\infty}, \ \|\prod\limits_{s=m_1+1}^{m}(z_{s}-z)\|_{\infty}\}.
	\end{equation*} 	
	Since the zero points of $g$ tend to the zero points of $f$ when $g$ tends to $f$ in the supremum norm, there exists a positive number $\delta<\frac{\epsilon}{2}$ such that if $g\in \textrm{Hol}_0(\overline{\mathbb{D}})$ with $\|f-g\|\le \delta$,
	then each $\widetilde{\alpha_{k}}$ is sufficiently close to $\alpha_{k}$ such that
	\begin{equation}\label{1}
		\|B_f-\widetilde{B}_f\|_{\infty}\le\frac{\epsilon}{2C_2(M^2+4M)},
	\end{equation}
	and each $\widetilde{z_{s}}$ is sufficiently close to $z_s$ such that
	\begin{align}
		&\|G^{inn}_{F_1}-\prod\limits_{s=1}^{m_1}(z_{s}-z)\|_{\infty}\le\frac{\epsilon}{2C_2(M^2+4M)},  \label{2}\\
		&\|\widehat{G^{inn}_{F_1}}-\prod_{s=1}^{m_1}\frac{|z_{s}|}{\overline{z_{s}}}(1-\overline{z_{s}}z)\|_{\infty}\le\frac{\epsilon}{2C_2(M^2+4M)},  \label{3}\\
		&\|G^{out}_{F_1}-\prod\limits_{s=m_1 +1}^{m}(z_{s}-z)\|_{\infty}\le\frac{\epsilon}{2C_2(M^2+4M)}. \label{4}
	\end{align}
	Notice that $z_s\in\partial \mathbb{D}$, for $s=1, \ldots, m$. Then, it is easy to see that
	\begin{align*}
		\frac{|z_{s}|}{\overline{z_{s}}}(1-\overline{z_{s}}z)=\frac{1}{\overline{z_{s}}}(\overline{z_{s}}z_s-\overline{z_{s}}z)=z_s-z,
	\end{align*}
	and consequently we could rewrite the inequality (\ref{3}) as
	\begin{align}\label{5}
		\|\widehat{G^{inn}_{F_1}}-\prod_{s=1}^{m_1}(z_s-z)\|_{\infty}\le\frac{\epsilon}{2C_2(M^2+4M)}.
	\end{align}
Following from  the inequalities (\ref{2}), (\ref{5}) and (\ref{4}), one can see that
\[
\|G^{inn}_{F_1}\|\le \|\prod\limits_{s=1}^{m_1}(z_{s}-z)\|_{\infty}+\frac{\epsilon}{2C_2(M^2+4M)}\le M,
\]
\[
\|\widehat{G^{inn}_{F_1}}\|_{\infty}\le M \ \ \ \text{and} \ \ \ \|G^{out}_{F_1}\|_{\infty}\le M.
\]
Furthermore, together with inequality (\ref{1}), we have
\begin{align*}
	&\|F_1-\widehat{G^{inn}_{F_1}}G^{out}_{F_1}\|_{\infty}\\
	=&\|\prod\limits_{s=1}^{m_1}(z_{s}-z)\prod\limits_{s=m_1+1}^{m}(z_{s}-z)-\widehat{G^{inn}_{F_1}}G^{out}_{F_1}\|_{\infty}\\
	\le&\|\prod\limits_{s=1}^{m_1}(z_{s}-z)\prod\limits_{s=m_1+1}^{m}(z_{s}-z)-\prod\limits_{s=1}^{m_1}(z_{s}-z)G^{out}_{F_1}\|_{\infty}+\|\prod\limits_{s=1}^{m_1}(z_{s}-z)G^{out}_{F_1}-\widehat{G^{inn}_{F_1}}G^{out}_{F_1}\|_{\infty}\\
	\le&\|\prod\limits_{s=1}^{m_1}(z_{s}-z)\|_{\infty}\|\prod\limits_{s=m_1+1}^{m}(z_{s}-z)-G^{out}_{F_1}\|_{\infty}+\|G^{out}_{F_1}(z)\|_{\infty}\|\prod\limits_{s=1}^{m_1}(z_{s}-z)-\widehat{G^{inn}_{F_1}}\|_{\infty}\\
	\le& \frac{2M\epsilon}{2C_2(M^2+4M)}
\end{align*}
and similarly,
\begin{align*}
	&\|F_1-{G^{inn}_{F_1}}G^{out}_{F_1}\|_{\infty}\\
	\le&\|\prod\limits_{s=1}^{m_1}(z_{s}-z)\|_{\infty}\|\prod\limits_{s=m_1+1}^{m}(z_{s}-z)-G^{out}_{F_1}\|_{\infty}+\|G^{out}_{F_1}(z)\|_{\infty}\|\prod\limits_{s=1}^{m_1}(z_{s}-z)-{G^{inn}_{F_1}}\|_{\infty}\\
	\le& \frac{2M\epsilon}{2C_2(M^2+4M)}.
\end{align*}
Consequently,
	\begin{align*}
		&\|B_f F_1-\widetilde{B}_f G^{inn}_{F_1} G^{out}_{F_1}\|_{\infty}\\
		\le&\|B_f F_1-B_f G^{inn}_{F_1} G^{out}_{F_1}\|_{\infty}+\|B_f G^{inn}_{F_1} G^{out}_{F_1}-\widetilde{B}_f G^{inn}_{F_1} G^{out}_{F_1}\|_{\infty}\\
		\le&\|B_f\|_{\infty}\|F_1-G^{inn}_{F_1} G^{out}_{F_1}\|_{\infty}+\|G^{inn}_{F_1} G^{out}_{F_1}\|_{\infty}\|B_f-\widetilde{B}_f\|_{\infty}\\
		=&\|F_1-G^{inn}_{F_1} G^{out}_{F_1}\|_{\infty}+\|G^{inn}_{F_1}\|_{\infty} \|G^{out}_{F_1}\|_{\infty}\|B_f-\widetilde{B}_f\|_{\infty}\\
		\le&\frac{(M^2+2M)\epsilon}{2C_2(M^2+4M)}.
	\end{align*}
In addition, we have
	\begin{align*}
		f-g&=B_f F_1 F_0-B_g \widehat{G^{inn}_{F_1}}G^{out}_{F_1}G_0\\
		&=(B_f F_1 F_0-B_g \widehat{G^{inn}_{F_1}}G^{out}_{F_1}F_0)+(B_g \widehat{G^{inn}_{F_1}}G^{out}_{F_1}F_0-B_g \widehat{G^{inn}_{F_1}}G^{out}_{F_1}G_0)\\
	\end{align*}
	then
	\begin{align*}
		&\| \widehat{G^{inn}_{F_1}}G^{out}_{F_1}F_0- \widehat{G^{inn}_{F_1}}G^{out}_{F_1}G_0\|_{\infty}\\
		=&\|B_g(\widehat{G^{inn}_{F_1}}G^{out}_{F_1}F_0- \widehat{G^{inn}_{F_1}}G^{out}_{F_1}G_0)\|_{\infty}\\
		=&\|(f-g)-(B_f F_1 F_0-B_g \widehat{G^{inn}_{F_1}}G^{out}_{F_1}F_0)\|_{\infty}\\
		\le&\|f-g\|_{\infty}+\|F_0\|_{\infty}\|B_f F_1-B_g \widehat{G^{inn}_{F_1}}G^{out}_{F_1}\|_{\infty}\\
		\le& \delta+C_2\cdot \frac{(M^2+2M)\epsilon}{2C_2(M^2+4M)} \\
		\le&\frac{\epsilon}{2}+\frac{(M^2+2M)\epsilon}{2(M^2+4M)}.
	\end{align*}
	Therefore,
	\begin{align*}
		&\|F_f-F_g\|_{\infty}\\
		=&\|F_1 F_0-\widehat{G^{inn}_{F_1}}G^{out}_{F_1}G_0\|_{\infty}\\
		=&\|F_1 F_0-\widehat{G^{inn}_{F_1}}G^{out}_{F_1}F_0+\widehat{G^{inn}_{F_1}}G^{out}_{F_1}F_0-\widehat{G^{inn}_{F_1}}G^{out}_{F_1}G_0\|_{\infty}\\
		\le&\|F_0\|_{\infty}\|F_1 -\widehat{G^{inn}_{F_1}}G^{out}_{F_1}\|_{\infty}+\|\widehat{G^{inn}_{F_1}}G^{out}_{F_1}F_0-\widehat{G^{inn}_{F_1}}G^{out}_{F_1}G_0\|_{\infty}\\
		\le&C_2\cdot \frac{2M\epsilon}{2C_2(M^2+4M)}+\frac{\epsilon}{2}+\frac{(M^2+2M)\epsilon}{2(M^2+4M)}\\
		=&\epsilon.
	\end{align*}
	So $f\to g$ in the supremum norm implies $F_f\to F_g$ in the supremum norm, which means the map $Q_{out}:(\textrm{Hol}_n(\overline{\mathbb{D}}), \|\cdot\|_{\infty})\rightarrow (\mathfrak{F}, \|\cdot\|_{\infty})$ is continuous.

{\bf (2) The discontinuity of $Q_{inn}:(\textrm{Hol}_0(\overline{\mathbb{D}}), \ \|\cdot\|_{\infty})\rightarrow (\mathfrak{I}, \|\cdot\|_{\infty})$}	

	However, the map $Q_{inn}:(\textrm{Hol}_0(\overline{\mathbb{D}}), \|\cdot\|_{\infty})\rightarrow (\mathfrak{I}, \|\cdot\|_{\infty})$ is not continuous in the supremum norm. Here we give an example to show the discontinuity. Let $f_t(z)=t-z$ for $t\in[\frac{1}{2},2]$. Then $f_t$ is a continuous path in $\textrm{Hol}(\overline{\mathbb{D}})$. The inner part of $f_1$ is $Q_{inn}(f_1)=1$ and the inner part of $f_t$, for $t\in[\frac{1}{2},1)$, is $Q_{inn}(f_t)=\frac{t-z}{1-tz}$. It is easy to see that as $t\to 1$,
	\[
	\|f_t-f_1\|_{\infty}=|1-t|\rightarrow 0
	\]
	but	
	\[
	\|Q_{inn}(f_t)-Q_{inn}(f_1)\|_{\infty}=\|\frac{t-z}{1-tz}-1\|_{\infty}=2\nrightarrow 0.
	\]

{\bf (3) The continuity of  $Q_{inn}:(\textrm{Hol}_n(\overline{\mathbb{D}}), \|\cdot\|_{\infty})\rightarrow (\mathfrak{I}, \|\cdot\|_{2})$}
	
Next, we want to prove that $Q_{inn}:(\textrm{Hol}_n(\overline{\mathbb{D}}), \|\cdot\|_{\infty})\rightarrow (\mathfrak{I}, \|\cdot\|_{2})$ is continuous in the $H^{2}$-norm, which is different from the case in the supremum norm.

Recall that $Q_{inn}(g)=B_g=\widetilde{B}_{f}\widetilde{B}_{F}$ and $Q_{inn}(f)=B_f$. Notice that if $g$ tends to $f$ in the supremum norm, then each $\widetilde{\alpha_{k}}$ is sufficiently close to $\alpha_{k}$. More precisely, there exists a positive number $\delta'<\frac{\epsilon}{2}$ such that if $g\in \textrm{Hol}_0(\overline{\mathbb{D}})$ with $\|f-g\|\le \delta'$, we have
\begin{equation}\label{6}
\|B_f-\widetilde{B}_f\|_{2}\le\frac{\epsilon}{2}
\end{equation}
and by $z_s\in \mathbb{T}$ for $s=1,2,\ldots,m_1$,
\begin{equation}\label{7}
|\prod\limits_{s=1}^{m_1}|\widetilde{z_{s}}|-1|=|\prod\limits_{s=1}^{m_1}|\widetilde{z_{s}}|-\prod\limits_{s=1}^{m_1}|z_{s}||\le\frac{\epsilon^2}{8}.
\end{equation}

Write $\widetilde{B}_{F}$ as follows
\[
\widetilde{B}_{F}(z)=\sum_{k=0}^{\infty}c_k z^k,
\]
where $c_k$ is the $k$-th Taylor coefficient of $\widetilde{B}_{F}$.

Since $\widetilde{B}_{F}(z)=\prod_{s=1}^{m_1}\frac{\overline{\widetilde{z_{s}}}}{|\widetilde{z_{s}}|}\frac{(\widetilde{z_{s}}-z)}{(1-\overline{\widetilde{z_{s}}}z)}$ is an inner function, it is not difficult to see that
\begin{align*}
\|\widetilde{B}_{F}-1\|_2 ^2&=\|\sum_{k=0}^{\infty}c_k z^k-1\|_2 ^2\\
&=\|c_0 -1+\sum_{k=1}^{\infty}c_k z^k\|_2 ^2\\
&\le |c_0 -1|^2+|\sum_{k=1}^{\infty}c_k|^2\\
&=|c_0 -1|^2+\sum_{k=0}^{\infty}|c_k| ^2-|c_0|^{2}\\
&=|\prod\limits_{s=1}^{m_1}|\widetilde{z_{s}}|-1|^2+1-\left(\prod\limits_{s=1}^{m_1}|\widetilde{z_{s}}|\right)^2 \\
&\le \frac{\epsilon^4}{64}+1-(1-\frac{\epsilon^2}{8})^2 \\
&=\frac{\epsilon^2}{4},
\end{align*}
that is
\begin{equation}\label{8}
\|\widetilde{B}_{F}-1\|_2 <\frac{\epsilon}{2}.
\end{equation}
Consequently, it follows from inequalities (\ref{6}) and (\ref{8}) that
\begin{align*}
	\|B_f-B_g\|_{2}&=\|B_f-\widetilde{B}_{f}\widetilde{B}_{F}\|_2\\
	&=\|B_f-B_f \widetilde{B}_{F}+B_f \widetilde{B}_{F}-\widetilde{B}_{f}\widetilde{B}_{F}\|_2\\
	&\le\|B_f-B_f \widetilde{B}_{F}\|_2+\|B_f \widetilde{B}_{F}-\widetilde{B}_{f}\widetilde{B}_{F}\|_2\\
	&=\|\widetilde{B}_{F}-1\|_2+\|B_f-\widetilde{B}_f\|_{2}\\
	&<\frac{\epsilon}{2}+\frac{\epsilon}{2}\\
	&=\epsilon
\end{align*}
So $f\to g$ in the $H^{2}$-norm implies $B_f\to B_g$ in the $H^{2}$-norm, which means the map $Q_{inn}:(\textrm{Hol}_n(\overline{\mathbb{D}}), \|\cdot\|_{\infty})\rightarrow (\mathfrak{I}, \|\cdot\|_{2})$ is continuous.

{\bf (4) The continuity of $Q_{out}:(\textrm{Hol}_n(\overline{\mathbb{D}}), \|\cdot\|_{\infty})\rightarrow (\mathfrak{F}, \|\cdot\|_{2})$}

At the end, since $Q_{out}:(\textrm{Hol}_n(\overline{\mathbb{D}}), \|\cdot\|_{\infty})\rightarrow (\mathfrak{F}, \|\cdot\|_{\infty})$ is continuous in the supremum norm, it is obvious that $Q_{out}:(\textrm{Hol}_n(\overline{\mathbb{D}}), \|\cdot\|_{\infty})\rightarrow (\mathfrak{F}, \|\cdot\|_{2})$ is also continuous in the $H^{2}$-norm.
\end{proof}

\section*{Declarations}

\noindent \textbf{Ethics approval}

\noindent Not applicable.

\noindent \textbf{Competing interests}

\noindent The author declares that there is no conflict of interest or competing interest.

\noindent \textbf{Authors' contributions}

\noindent All authors contributed equally to this work.

\noindent \textbf{Funding}

\noindent There is no funding source for this manuscript.

\noindent \textbf{Availability of data and materials}

\noindent Data sharing is not applicable to this article as no data sets were generated or analyzed during the current study.

\end{document}